\documentclass[11pt,reqno,oneside]{amsart}
\usepackage[a4paper,top=3.5cm,bottom=3.5cm,left=3cm,right=3cm]{geometry}
\usepackage[utf8]{inputenc}
\usepackage[english]{babel}
\usepackage{amsfonts}
\usepackage{amsmath}
\usepackage{amsthm}
\usepackage{amssymb}
\usepackage{comment}
\usepackage{float}
\usepackage{textcomp}
\usepackage{faktor}
\usepackage{setspace}
\usepackage[dvipsnames]{xcolor}
\usepackage{graphicx}
\usepackage{fancybox}
\usepackage{mathrsfs}
\usepackage{mathtools}
\usepackage{tikz}
\usepackage{centernot}
\usepackage{esint}
\usepackage{algpseudocode}
\usepackage{framed}
\usepackage{hyperref}
\usepackage[bottom]{footmisc}
\usepackage[square, sort, numbers]{natbib}
\usepackage[T1]{fontenc}
\usepackage{mathtools}
\usepackage{graphicx}
\usepackage{epstopdf}
\usepackage{listings}
\usepackage{fancyhdr}
\usepackage{afterpage}
\usepackage{ dsfont }
\usepackage{cleveref}
 \usepackage[foot]{amsaddr}
\DeclarePairedDelimiter{\spam}{ \langle } { \rangle }

\newcommand{\sfrac}[2]{ 
	{\raise0.8ex\hbox{$#1$} \!\mathord{\left/ 
			{\vphantom {#1 #2}}\right.\kern-\nulldelimiterspace} 
		\!\lower0.8ex\hbox{$#2$}} 
}

\newtheorem{teo}{Theorem}[section]
\newtheorem{lema}[teo]{Lemma}
\newtheorem{co}[teo]{Corollary}
\newtheorem{prop}[teo]{Proposition}

\newtheorem{hypothesis}[teo]{Hypothesis}

\theoremstyle{definition}
\newtheorem{defin}[teo]{Definition}

\newtheorem{obs}[teo]{Remark}

\newcommand{\R}{\mathbb{R} }
\newcommand{\E}{\mathbb{E} }

\newcommand{\F}{\mathcal{F}}
\newcommand{\calG}{\mathcal{G}}
\newcommand{\calA}{\mathcal{A}}
\newcommand{\calU}{\mathcal{U}}
\newcommand{\calZ}{\mathcal{Z}}

\newcommand{\calS}{\mathcal{S}}

\newcommand{\calH}{\mathcal{H}}

\newcommand{\hatu}{\hat{u}}

\newcommand{\hatc}{\hat{c}}

\newcommand{\tZ}{\widetilde{\mathcal{Z}}}

\renewcommand{\t}{t \in [0,T]}

\newcommand\freefootnote[1]{%
  \begingroup
  \renewcommand\thefootnote{}\footnote{#1}%
  \addtocounter{footnote}{-1}%
  \endgroup
}
\title{Lifting of Volterra processes:\\ optimal control in UMD Banach spaces}

\author{Giulia di Nunno$^{\dagger, \ast}$ \and Michele Giordano$^{\ddagger}$ \\ June 22nd 2023}
\begin{document}
	\allowdisplaybreaks
\maketitle

\numberwithin{equation}{section}

\freefootnote{$^{\dagger}$Department of Mathematics, University of Oslo, P:O: Box 1053 Blindern, N-0316 Oslo, Email: giulian@math.uio.no.}
\freefootnote{$^{\ast}$Department of Business and Management Science, NHH Norwegian School of Economics, Helleveien 30, N-5045 Bergen.}
\freefootnote{$^{\ddagger}$Department of Mathematics, University of Oslo, P:O: Box 1053 Blindern, N-0316 Oslo, Email: michelgi@math.uio.no}
\vspace{-1 cm}

%%%%%%%%%%%%%%%%%%%%%%%%%%%%%%%%%%%%%%%%%%%%%%%%%%%%%%%%%%%%%%%%%%%%%%%%%%%%%%%%%%%%%%%%%%%

\begin{abstract}
	\noindent We study a stochastic control problem for a Volterra-type controlled forward equation with past dependence obtained via convolution with a deterministic kernel. To be able to apply dynamic programming to solve the problem, we lift it to infinite dimensions and we formulate a UMD Banach-valued Markovian problem, which is shown to be equivalent to the original finite-dimensional non-Markovian one. We characterize the optimal control for the infinite dimensional problem and show that this also characterizes the optimal control for the finite dimensional problem.
\end{abstract}
 Hamilton Jacobi Bellman; Optimal control; UMD Banach space; Markovian Lift; 
\section{Introduction}

We intend to minimize a performance functional of the form
\begin{equation}\label{J(u)}
	J(t,x,u)=\E\left[\int_t^T F(\tau,X^u_\tau,u_\tau)d\tau+G(X^u_T) \right],
\end{equation}
where $\t$, and $x$ is given in the controlled Volterra-type dynamics of the process $X_\tau^u$:
\begin{align}
	X^u(\tau)&=x(\tau)+\int_t^\tau K(\tau-s)\Big[\beta(s,X^u(s))+\sigma(s,X^u(s))R(s,X^u(s),u(s))\Big]ds\nonumber \\
	&\quad +\int_t^\tau K(\tau-s) \sigma(s,X^u(s))dW(s)\label{X_t}.
\end{align}
Here $x:[0,T]\longrightarrow \R $, $\beta:[0,T]\times \R \longrightarrow \R$, $\sigma:[0,T]\times \R \longrightarrow \R$, $R:[0,T]\times \R \times \calU \longrightarrow \R$, and the convolution kernel $K:[0,T]\longrightarrow\R^+$ are all measurable mappings on which additional hypothesis will be stated later on, and $u:[0,T]\times \Omega \longrightarrow \calU\subset \R$ is an admissible control. Also, $W$ is a real-valued Brownian motion on a complete filtered probability space $(\Omega, \mathcal F, \mathbb P)$.

Stochastic control problems as \eqref{J(u)}-\eqref{X_t} appear, e.g. when studying optimal advertising strategies (see e.g. \cite{Antony} for the case of Volterra dynamics in \eqref{X_t} and \cite{GMS2009, GM2006, LZh2020} for the case of delay). Other cases of applications are found in electrodynamics \cite{electrodynamics} and in epidemiology \cite{epidemiological}.
 \noindent When dealing with such problems, one cannot directly apply a \textit{dynamic programming principle} (DPP) in view of the non Markovianity of the framework. While in some particular cases is still possible to derive the DPP also for Volterra forward dynamics (see \cite{Pham,Possamai, Bonaccorsi}), most authors approached the general problem by means of a maximum principle (see, e.g., \cite{DNG1,Yong1,Yong2,Oksendal1,Oksendal2} and references therein). Even though the maximum principle approach might seem practical, one usually has to impose regularity conditions on both the drift and volatility which are not always easy to satisfy.
In this paper, thanks to the developments on the \emph{lift} theory for Volterra processes (see \cite{Pham,CT,CTMulti,Note,Antony,Bonaccorsi}) we aim to move the stochastic control problem \eqref{J(u)}-\eqref{X_t} to an infinite dimensional UMD-Banach setting and solve the newly formulated problem by means of DPP. 

The main purpose of this lifting approach is to recover the Markov property for the forward process \eqref{X_t} which, in turns, allows us to derive a DPP in terms of the \emph{Hamilton-Jacobi-Bellman} (HJB) equations. In fact, one can show that solving the lifted problem is equivalent to solving the original one, with the fundamental difference that, by moving to an infinite dimensional setting, we work in a Markovian framework. Focusing on Markovian lifts, we assume that the kernel $K$ can be represented as $K(t)=\langle g,\calS_t^*\nu \rangle_{Y\times Y^*} $, for $\calS_t^*$ a uniformly continuous semigroup acting on a Banach space $Y^*$, $\nu\in Y^*$, $g\in Y$ with $Y$ the pre dual of $Y^*$ and pairing $\langle \cdot, \cdot \rangle_{Y\times Y^*}$. Examples of kernels that satisfy this condition can be found both in \cite{CT, CTMulti, Antony} and in the last section of this paper.

Our goal is to find $\hatu$ such that, for all $\t$,
\begin{equation}\label{defin: control problem}
    J(t,x,\hatu)=\inf_{u\in\mathds A}J(t,x,u),
\end{equation}
with $J(t,x,u)$ as in \eqref{J(u)} and for $u$ belonging to some admissible control set $\mathds A$ defined as 
$$\mathds A=\left\{u:[0,T]\times\Omega\longrightarrow \calU,\ \text{s.t.   $u$ is predictable} \right\},$$
where $\calU$ is a closed convex subset of $\R$ and the information flow is associated to the Brownian motion in \eqref{X_t}. Our approach consists in formulating a new infinite-dimensional Banach-valued optimization problem that can be shown to be equivalent to \eqref{J(u)}-\eqref{X_t} (in the sense that the optimal control $\hat u$ and optimal value $J(t,x,\hat u )$, $\t$ are the same of the original one) and then solve such infinite dimensional optimization problem, which is Markovian.
%The problem reformulated in the infinite dimension setting, follows the approach presented in \cite{FT,Masiero}, while its
The solution is achieved exploiting Malliavin calculus for \emph{unconditional martingale differences} (UMD) Banach spaces.

The Markovian lift to the infinite dimensional setting that we present here was originally introduced in \cite{CT}, and then developed in \cite{CTMulti} for the multi-dimensional case, and in \cite{Note} for a Lévy drivers. Our work can be seen as a generalization of the case presented in \cite{Pham}, in which a kernel $K$ that can be expressed as the Laplace transform of a measure is considered, and where the performance functional \eqref{J(u)} is of linear-quadratic type. Our work differs from \cite{Pham} as we are able to consider a broader class of kernels and performance functionals thanks to the nature of the lift we apply.

The present work introduces an element of novelty also with respect to infinite dimensional stochastic control. Indeed, we consider a setting which is different from both the ones presented in \cite{FT} and \cite{Masiero}. In \cite{FT} the authors consider a Hilbert valued forward controlled process, whereas in \cite{Masiero} the forward process has values in a general Banach space, but with a volatility term $\sigma$ not depending on $X$. Here we are able to take general volatility dynamics for the forward process and this will require to work in Banach spaces of the UMD type so to be able to apply Malliavin calculus techniques.

In the context of optimal control for lifted process, we also mention \cite{Bonaccorsi}. There, the authors follow an approach close to the one presented here. However, we remark that, in our framework, we are able to consider a wider class of kernel $K$ thanks to the nature of our lift, which allows us to work in UMD Banach spaces instead of Hilbert spaces. On the other side, in \cite{Bonaccorsi} the authors consider a forward equation driven by a Lévy process instead of a Brownian motion (like \eqref{X_t}). While the lift theory for Lévy-driven forward processes is available (see \cite{Note}), the optimal control of an infinite-dimensional Lévy-driven forward equation in the present setting is a topic for future research.

This paper is structured as follows: in Section 2 we  present some preliminary results both on the Gâteaux differentiability in general Banach spaces and on the lift for Volterra processes. We recall the essentials on UMD Banach spaces and some results of Malliavin calculus in this framework. In Section 3 we give an existence and continuity result  for the forward equation, and in Section 4 we introduce the backward equation and the Hamiltonian function associated with the lifted optimization problem. Here we present a solution method via HJB equations. To conclude, in Section 5 we present a problem of optimal consumption where we obtain a characterization of the optimal control $u$ via DPP.

\section{Some preliminary results}
We recall some useful notions and results that are used throughout the paper. Then we show how the Markovian lift is performed in the present context of stochastic control \eqref{J(u)}-\eqref{X_t}. Lastly, we introduce UMD Banach spaces and state some crucial results for Malliavin calculus in this setting.
We refer to \cite{FT} for the results on Gâteaux derivatives and Banach spaces, to \cite{CT, CTMulti,Note, Antony} for the ones concerning Markovian lifts, and to \cite{UMD, MalliavinBanach, VanNerv,Clark-Ocone} for the results concerning UMD Banach spaces.

\subsection{The class of Gâteaux differentiable functions}
For a mapping $F:U\longrightarrow V$, with $U,V$ two Banach spaces, we say that the \emph{directional derivative at $u\in U$ in the direction $h\in U$} is defined as
\begin{equation*}
	\nabla F(u;h):=\lim_{s\rightarrow 0}\frac{F(u+sh)-F(u)}{s},
\end{equation*}
whenever the limit exists in the topology of $V$. The mapping $F$ is said to be \emph{Gâteaux differentiable at the point $u$} if it has directional derivative at $u$ in every direction and there exists an element $\nabla F(u)$ in $L(U,V)$ such that $\nabla F(u;h)=\nabla F(u)h$ for every $h\in U$. We call $\nabla F(u)$ the \emph{Gâteaux derivative at $u$}.
\begin{defin}
	A mapping $F:U\longrightarrow V$ belongs to $\calG^1(U;V)$ if it is continuous, Gâteaux differentiable for all $u\in\calU$ and $\nabla F:U\longrightarrow L(U;V)$ is strongly continuous, i.e. the map $\nabla F(\cdot)h:U\longrightarrow V$ is continuous for every $h\in U$.
\end{defin}
\begin{obs}
	Let $U,V,Z$ be three Banach spaces and $F\in\calG^1(U,V)$. If $G\in\calG^1(V,Z)$, then $G(F)\in\calG^1(U,Z)$ and $\nabla(G(F))(u)=\nabla G(F(u))\nabla F(u)$.
\end{obs}
We also introduce the \emph{partial directional derivative} for a mapping $F:U\times V \longrightarrow Z$, with $U$, $V$, $Z$ Banach spaces as 
\begin{equation*}
    \nabla_u F(u,v;h) := \lim_{s\rightarrow 0 } \frac{F(u+sh,v)-F(u, v)}{h},
\end{equation*}
with $u,h \in U$, $v\in V$ and the limit taken in the topology of $Z$. We say that $F$ is \emph{partially Gâteaux differentiable with respect to $u$ at $(u,v) \in U\times V$} if there exists $\nabla_u F : U\times V\longrightarrow L(U,Z)$ such that $\nabla_u F(u,v;h) = \nabla_uF(u,v)h$ for all $h\in U$.
\begin{defin}
	We say that $F:U\times V \longrightarrow Z$ belongs to the class $\calG^{1,0}(U\times V;Z)$ if it is continuous, Gâteaux differentiable with respect to $u$, for all $(u,v)\in U\times V$ and $\nabla_uF:U\times V \longrightarrow L(U,Z)$ is strongly continuous.
\end{defin}
\noindent For $F$ depending on additional arguments, the definition above can be easily generalized.
\begin{lema}\label{LEMMA 2.3 FT} 
	Given $U,V,Z$ three Banach spaces, a continuous map $F:U\times V\longrightarrow Z$ belongs to $\calG^{1,0}(U\times V, Z)$ provided the following conditions hold:
	\begin{enumerate}
		\item The partial directional derivatives $\nabla_uF(u,v;h)$ exist at every point $(u,v)\in U\times V$ and in every directoin $h\in U$.
  		\item For every $(u,v)$ the mapping $h\longmapsto\nabla_uF(u,v;h)$ is continuous from $U$ to $Z$.
		\item For every $h$, the mapping $\nabla_u F(\cdot, \cdot;h):U\times V\longrightarrow Z$ is continuous.
	\end{enumerate}
\end{lema}
\noindent We are going to use the following parameter depending contraction principle to study the regular dependence of the solution of stochastic differential equations on their initial data.
\begin{prop}\label{PROP 2.4 FT}
	Let $U,V,Z$ be Banach spaces and let $F:U\times V\times Z \longrightarrow U$ a continuous mapping satisfying 
	\begin{equation*}
		|F(u_1,v,z)-F(u_2,v,z)|\leq \alpha|u_1-u_2|,
	\end{equation*}
	for some $\alpha\in[0,1)$ and every $u_1,u_2 \in U$, $v \in V$, $z \in Z$. Let $\phi(v,z)$ denote the unique fixed point of the mapping $F(\cdot,v,z):U\longrightarrow U$. Then $\phi:V\times Z\longrightarrow U$ is continuous. If, in addition $F\in\calG^{1,1,0}(U\times V\times Z, U)$, then $\phi\in\calG^{1,0}(V\times Z,U)$ and
	\begin{equation*}
		\nabla_v\phi(v,z)=\nabla_uF\big(\phi(v,z),v,z\big)\nabla_v\phi(v,z)+\nabla_vF(\phi(v,z),v,z).
	\end{equation*}
\end{prop}

\subsection{Lift approach to optimal control}

In the sequel, we exploit an infinite dimensional lift to reformulate the optimization problem \eqref{J(u)}-\eqref{X_t} in an infinite dimensional setting. Our first step is to rewrite $X^u$ in \eqref{X_t} in terms of a process $\calZ^u$ with values in a Banach space, using the lift procedure presented in \cite{CT}. Notice that we do not actually work in the affine framework of \cite{CT}, but the approach presented here is actually a particular case of the one introduced in \cite{Note}.

\begin{defin}\label{liftable}
    Let $Y$ be a Banach space with dual $Y^*$ and denote with $\langle\cdot,\cdot\rangle_{Y\times Y^{*}}$ the pairing between $Y$ and $Y^*$. We say that a kernel $K\in L^2_{loc}(\R_+,\R)$ is \emph{liftable} if there exist $g \in Y$, $\nu \in Y^*$ and a uniformly continuous semigroup $\calS_t^*$, $\t$ with generator $\calA^*$, acting on $Y^*$, such that
    \begin{itemize}
        \item  $K(t)=\langle g, \calS_t^*\nu\rangle_{Y\times Y^{*}}$
        \item $\calS_t^*\nu \in Y^*$ for all $t>0$ 
        \item $\int_0^t\|\calS_s^*\nu\|^2_{Y^*}ds<\infty$ for all $t>0$.
    \end{itemize}
    For notational simplicity we write $\spam{\cdot,\cdot}$ for $\spam{\cdot, \cdot}_{Y\times Y^{*}}$ when no confusion arises.
\end{defin}
%\vspace{2 mm}
\noindent From now on, we make the following assumption:
\begin{hypothesis}
    The kernel $K$ in \eqref{X_t} is liftable.
\end{hypothesis}
\noindent We rewrite $X^u$ as 
\begin{align*}
    X^u  (\tau)  &=x(\tau)+\int_t^\tau K(\tau-s)\Big[\beta(s,X^u(s))+\sigma(s,X^u(s))R(s,X^u(s),u(s))\Big]ds\\
    &\quad+\int_t^\tau K(\tau-s) \sigma(s,X^u(s))dW(s)\\
        &:= x(\tau) +\int_t^\tau K(\tau-s) dV^u(s),
\end{align*}
where 
    \begin{equation}\label{V(s)}
    dV^u(s) := \Big[\beta(s, X^u(s)) +\sigma(s, X^u(s))R(s,X^u(s),u(s))\Big] ds + \sigma(s,X^u(s))dW(s).     
    \end{equation}
\noindent Defining $\zeta$ as an element in $Y^*$ such that $x(\tau) =: \langle g,\calS_\tau^*\zeta\rangle$ we can now rewrite \eqref{X_t} as follows:
\begin{align}
    X^u(\tau) &= x(\tau)+\int_t^\tau K(\tau-s) dV^u(s)\nonumber\\
    &= \langle g, \calS_\tau^*\zeta\rangle + \int_t^\tau \langle g, \calS_{\tau-s}^*\nu\rangle dV^u(s)\nonumber\\
    &= \Bigg\langle g, \calS_\tau^*\zeta + \int_t^\tau \calS_{\tau-s}^*\nu dV^u(s)\Bigg \rangle\nonumber\\
    &=:\langle g, \calZ_\tau^u\rangle,\label{lift}
\end{align}
where $\calZ_\tau^u := \calS_\tau^* \zeta + \int_t^\tau \calS_{\tau-s}^*\nu dV^u(s)$. One can then check that $\calZ_\tau^u$ follows the dynamics:
\begin{equation}\label{Z_t}
    \calZ_\tau^u =\calS_t^* \zeta + \int_t^\tau \calA^*\calZ_s^u ds + \int_t^\tau \nu dV^u(s),
\end{equation}
In fact, we have that
\begin{align*}
    \int_t^\tau \calA^* \calZ^u_s ds &= \int_t^\tau \calA^* \left[ \calS_s^*\zeta + \int_t^s \calS^*_{s-v} \nu dV^u(v) \right]ds\\
    &= e^{\calA^*\tau }\zeta - e^{ \calA^*t}\zeta + \int_t^\tau \!\!\!\int_v^\tau \calA^* e^{\calA^*(s-v)}\nu ds dV^u(v)\\
    &=e^{ \calA^*\tau}\zeta - e^{ \calA^*t}\zeta + \int_t^\tau e^{\calA^*(\tau - v)}\nu dV^u(v) - \int_t^\tau \nu dV^u(v)\\
    &= \calZ_\tau^u - e^{ \calA^*t}\zeta - \int_t^\tau \nu dV^u(v),
\end{align*}
and, rearranging the terms we obtain \eqref{Z_t}.
\begin{obs}
    By defining $B(t, X^u(t), u(t)) := \beta(s, X^u(s))  \sigma(s, X^u(s))R(s,X^u(s), u(s))$, and exploiting \eqref{lift}, we actually get that the function $x(\tau)= \langle g, \calS_t^*\zeta \rangle$ is given by the expression
    \begin{align*}
        x(\tau) &= \E\left[X^u(\tau) - \int_t^\tau K(\tau-s) B(s,X^u(s), u(s)) ds\right]\\
            &= \left\langle g, \E\left[\calZ^u_\tau - \int_t^\tau \calS_{\tau-s}^*\nu B(s, X^u(s), u(s)) ds\right]\right\rangle
    \end{align*}    
\end{obs}
\noindent Set $\calZ_{\tau}^{u,g}:=\langle g, \calZ_{\tau}^u\rangle=X^u(\tau)$, and plug \eqref{V(s)} into \eqref{Z_t}, then we can rewrite \eqref{Z_t} in differential notation as 
\begin{equation}\label{dZ}
	d\calZ_{\tau}^u=\calA^*\calZ_{\tau}^u d{\tau}+\nu\Big(\beta({\tau},\calZ_{\tau}^{u,g})d\tau+\sigma({\tau},\calZ_{\tau}^{u,g})\big[R({\tau},\calZ_{\tau}^{u,g},u_{\tau})d{\tau}+dW_{\tau}\big]\Big),
\end{equation}
with $\calZ^u_t = e^{\calA^* t} \zeta := \zeta_t$. 
We also rewrite \eqref{Z_t} as 
\begin{equation}\label{Z_t rewritten}
	d\calZ_{\tau}^u=\calA^*\calZ_{\tau}^ud{\tau}+\nu\beta^g({\tau},\calZ_{\tau}^{u})d\tau+\nu\sigma^g({\tau},\calZ_{\tau}^{u})[R^g({\tau},\calZ_{\tau}^{u},u_{\tau})d\tau+dW_{\tau}],
\end{equation}
where
\begin{align*}
    \beta^g(s,\calZ_s^u)&:=\beta(s,\calZ_s^{u,g}) = \beta(s,\langle g, Z_s^u\rangle) = \beta(s,X(s)),\\
    \sigma^g(s,\calZ_s^u)&:=\sigma(s,\calZ_s^{u,g}) = \sigma(s,\langle g, Z_s^u\rangle) = \sigma(s,X(s)),\\
    R^g(s,\calZ_s^u,u_s)&:=R(s,\calZ_s^{u,g}, u_s) = R(s,\langle g, Z_s^u\rangle, u_s) = R(s,X(s),u_s).
\end{align*}
\noindent We are going to discuss existence and uniqueness results for equation \eqref{dZ} in Section \ref{sec: forward}.
\begin{obs}\label{remark brownian}
        We point out that the term
        \begin{equation}\label{volatility}
            \int_0^\tau \nu \sigma^g(s,\calZ_s)dW_s,
        \end{equation}
         in \eqref{Z_t rewritten} can be regarded in two different ways. 
         On the one hand, it can be seen as the element of $Y^*$:
        \begin{equation*}
            \nu\left(\int_0^\tau\sigma^g(s,\calZ_s)dW_s\right),
        \end{equation*}
        where the integration of $\int_0^\tau\sigma^g(s,\calZ_s)dW_s=\int_0^\tau\sigma(s,X(s))dW_s$ is done on $\R$ and then lifted to $Y^*$ by multiplying it by $\nu$.
        On the other hand, by writing \eqref{volatility} as
        \begin{equation*}
            \int_0^\tau\sigma^g(s,\calZ_s) d(\nu W_s),
        \end{equation*}
        we have that $\nu W_s$ can be considered as a cylindrical Wiener process on 
        $\mathbb{H}^\nu:=\left\{\nu x, \ x\in\R \right\}$, 
        which is a Hilbert space with the scalar product $\langle \cdot, \cdot\rangle_{\mathbb{H}^{\nu}}:=\|\nu\|_{Y^*}\langle \cdot, \cdot\rangle_\R$. In this case we also see that $\mathbb{H}^\nu\subsetneq Y^*$.    
\end{obs}

In \cref{sec: forward} we are going to provide sufficient conditions that guarantee the existence of a solution of \eqref{Z_t}-\eqref{dZ}. Due to the nature of the lift and identification \eqref{lift}, this will, in turn, provide sufficient conditions also for the existence of a solution to \eqref{X_t}.

\begin{obs}\label{HP Cuchiero}
	From \cite{CT,CTMulti} we see that we could perform the lift also under weaker hypothesis, by taking a subspace $Z\subset Y$ with their relative duals $Y^*\subset Z^*$ such that:
	\begin{itemize}
		\item $Z$ and $Y$ are Banach spaces $Z\subset Y$ and $Z$ embeds continuously into $Y$.
		\item The semigroup $\calS^*$ with generator $\calA^*$ acts in a strongly continuous way on $Y^*$ and $Z^*$ with respect to the respective norm topologies.
		\item The map $\calZ \longmapsto \calS^*_t\calZ$ is weak-* continuous on $Y^*$ and on $Z^*$ for every $t\geq 0$.
		\item The pre-adjoint operator of $\calA^*$, generates a strongly continuous semigroup on $Z$ with respect to the respective norm topology (but not necessarily on $Y$). 
	\end{itemize}
	In this case every kernel of the form $K(t)=\langle g, \calS_t^*\nu \rangle $ with $\nu \in Z^*$ and $S_t^*\nu\in Y^*$ is liftable. While this setting would allow to work with a wider class of kernels, we would not be able to formulate the HJB equations. This is due to the fact that, when considering a kernel $K(t)=\langle g, \calS_t^*\nu\rangle$ with $\nu\in Z^*$, some of the inner products in the definition of the Hamilton-Jacobi-Bellman equation \eqref{HJB}, would not be well defined. Being the goal of this work a control problem, we restrict ourselves to the case $\nu\in Y^*$, as originally stated.
\end{obs}

\noindent In a similar fashion to what we did for \eqref{X_t}, recalling that $X^u(\tau):=\calZ_\tau^{u,g}$, we rewrite the performance functional \eqref{J(u)} so to make its dependence from the lifted process $\calZ^u_\tau$ explicit:
\begin{align}
	J(t,x,u)&=\E\left[\int_t^TF(\tau,\calZ^{u,g}_\tau,u_\tau)d\tau+G(\calZ^{u,g}_T)\right]\nonumber \\
 &:=\E\left[\int_t^TF^g(\tau,\calZ_\tau^u,u_\tau)d\tau+G^g(\calZ_T^u)\right]:=J^g(t,\zeta,u),\label{equivalence J(u)}
\end{align}
where the functions $F:[0,T]\times \R\times \calU\longrightarrow\R$, $G:\R\longrightarrow\R$ are \emph{lifted} to the functions 
\begin{align*}
    F^g&:[0,T]\times Y^*\times\calU\longrightarrow\R,\\
    G^g&:Y^*\longrightarrow\R
\end{align*}
where $Y^*$ is the Banach space associated to the liftable kernel, see Definition \ref{liftable}. The lifted maps $F^g$ and $G^g$ are 
\begin{align*}
    F^g(\cdot,\calZ^u_\tau,\cdot)&:=F(\cdot,\calZ^{u,g}_\tau,\cdot)=F(\cdot,\spam{g,\calZ^u_\tau},\cdot)=F(\cdot, X^u(\tau), \cdot),\\
    G^g(\calZ^u_\tau)&:=G(\calZ^{u,g}_\tau)=G(\spam{g,\calZ^u_\tau})=G( X^u(\tau)).
\end{align*}
The stochastic optimal control problem \eqref{J(u)}-\eqref{defin: control problem} is then lifted to
\begin{equation}\label{optimal tilde J(u)}
	J^g(t,\zeta,\hatu)=\inf_{u\in\mathds A}J^g(t,\zeta,u)=\inf_{u\in\mathds A}\E\left[\int_t^TF^g(\tau,\calZ_\tau^u,u)d\tau+G^g(\calZ_T^u)\right],
\end{equation}
where the process $\calZ^u$ takes values in the Banach space $Y^*$, and where the dynamics for the controlled process are given by \eqref{Z_t}-\eqref{dZ}. Notice that, while the performance functional has not changed, we write $J^g$ instead of $J$ in order to highlight the dependence on $\calZ^u_t$ instead of $X^u(t)$, as underneath there is a passage from finite to infinite dimensions. 
Indeed, this change of notation embodies a crucial change of framework from a finite to an infinite dimensional setting, allowing us to move from functions $\beta$, $\sigma$, $R$, $F$ and $G$ taking values from $\R$ to new functions $\beta^g$, $\sigma^g$, $R^g$, $F^g$ and $G^g$ that now take values from $Y^*$.
This lift allows us to consider a new optimization problem, written on a space which is not the original one. Nonetheless, we have that $J(t,x,u)=J^g(t,\zeta,u)$ for $\t$, $u\in\calU$. Also, being $g$ fixed and only depending on the kernel representation, finding the pair $(\hatu, \calZ^{\hatu})$ that minimizes \eqref{Z_t rewritten}-\eqref{optimal tilde J(u)} is equivalent to finding the pair $(\hatu,X^{\hatu})$ that solves \eqref{J(u)}-\eqref{X_t}.

\subsection{UMD Banach spaces}

In the sequel we use techniques of Malliavin calculus on the space $Y^*$. For this, we assume:
\begin{hypothesis}\label{UMD}
	The space $Y^*$ is a unconditional martingale differences (UMD) Banach space.
\end{hypothesis}
\noindent For convenience we report here below the essentials on UMD Banach spaces.
\begin{defin}
	Let $(M_n)_{n=1}^N$ be a Banach-space valued martingale, the sequence $d_n=M_{n+1}-M_n$ is called the martingale difference sequence associated with $(M_n)_{n=1}^N$. A Banach space $E$ is said to be a $UMD_p$ $(1<p<\infty)$, space if there exists a constant $\beta$ such that for all $E$-valued $L^p$-martingale difference sequences $(d_n)_{n=1}^N$ we have 
	$$\E\left \|\sum_{n=1}^{N}\epsilon_nd_n\right \|^p\leq \beta^p \E\left\|\sum_{n=1}^N d_n\right\|^p,$$
	where $\epsilon_n \in \R$ for all $n$ and $|\epsilon_n|=1.$
	Thanks to \cite{VanNerv} we also know that, if a Banach space $E$ is $UMD_p$ for some $1<p<\infty$, then $E$ is a $UMD_p$ Banach space for all $p\in(1,\infty)$, and we simply call it a \emph{UMD Banach space}. 
\end{defin}
\noindent In the context of stochastic analysis in Banach spaces, martingale difference sequences provide a substitute for orthogonal sequences. In the following parts, we will see that this hypothesis is not very restrictive, as the UMD Banach spaces include all Hilbert spaces, $L^q$ spaces for $q\in (1,\infty)$, reflective Sobolev spaces and many others, thus allowing us to consider a wide class of liftable kernels. In our framework, the process $\calZ^u$ takes values in a UMD Banach space whenever we consider, for example, a shift operator or a quasi-exponential kernel or a kernel that can be expressed as the Laplace transform of a measure with density in $L^q([0,\infty))$, $q\in(1,\infty)$.

Assuming that $Y^*$ is UMD, allows us to define the Malliavin derivative operator $D$ on $L^p(\Omega, Y^*)$. From \cite[Proposition 2.5]{MalliavinBanach}, we know that $D$ is a closed operator and we denote with $\mathbb{D}^{1,p}(Y^*)$ the closure of the domain.

For the results on UMD Banach spaces exploited in the following parts, we refer to \cite{BDG} for the BDG inequality, \cite{Fubini} for the Fubini Theorem and to \cite{MalliavinBanach} for general Malliavin calculus results. In this framework we will also use a Clark-Okone formula for UMD Banach spaces (see \cite{Clark-Ocone}) and the following chain rule linking the Malliavin derivative and the Gâteaux derivative (see \cite{MalliavinBanach})
\begin{prop}\label{ChainRule}
	Let $E$ be a UMD Banach space and let $p\in(1,\infty)$. Suppose that $\varphi\in \calG^1(E,E)$. If $F\in\mathbb{D}^{1,p}(E)$, then $\varphi(F)\in\mathbb{D}^{1,p}(E)$ with 
	\begin{equation*}
	D(\varphi(F))=\nabla\varphi(F)DF.
	\end{equation*}
\end{prop}

\section{The optimal control problem}

We are now interested in solving the lifted optimal control problem \eqref{optimal tilde J(u)}, where the process $\calZ$ follows the controlled dynamics given by
\begin{equation}\label{Z_t controlled}
	\begin{cases}
	d\calZ^u_\tau=&\calA^*\calZ_\tau^ud\tau+\nu\beta^g(\tau,\calZ^u_\tau)d\tau+\nu\sigma^g(\tau,\calZ^u_\tau)R^g(\tau,\calZ^u_\tau,u_\tau)d\tau\\
	&+\nu\sigma^g(\tau,\calZ^u_\tau)dW_\tau,\\
	\calZ^u_t=&\zeta_t.
	\end{cases}
\end{equation}
For our results to hold, we add some Hypothesis on $R$, which directly translates into hypothesis on $R^g$.
\begin{hypothesis}\label{HP5.1 Mas}
	$R:[0,T]\times \R\times \calU\longrightarrow \R$ is measurable and  $\|R(\tau,x,u)\|_{\R}\leq K_R$ for a suitable positive constant $K_R>0$ and every $\tau\in[0,T]$, $x\in \R$, $u\in\calU$.
\end{hypothesis}
\noindent In order to find the optimal value $J(\hat u )$, we associate the following partially coupled system of forward-backward equations
\begin{equation}\label{FBSDE}
\begin{cases}
d\calZ_\tau&=\calA^*\calZ_\tau d\tau+\nu\beta^g(\tau,\calZ_\tau)d\tau+\nu\sigma^g(\tau,\calZ_\tau)dW_\tau,\ \  \tau\in[t,T],\\
\calZ_t&=\zeta_t,\\
dp_\tau&=-\calH(\tau,\calZ_\tau,q_\tau)d\tau+q_\tau \nu dW_\tau, \quad \tau\in[t,T],\\
p_T&=G(\calZ_T),
\end{cases}
\end{equation}
to \eqref{Z_t controlled}. Here above $\calH:[0,T]\times Y^*\times Y^{**}\longrightarrow \R$ is the Hamiltonian function defined as
\begin{equation}\label{Hamilton}
\calH(t,z,\xi)=\inf_{u\in\calU}\left[F^g(t,z,u)+\xi \nu R^g(t,z,u)\right].
\end{equation} 
Notice that the control $u$ only appears in the Hamiltonian functional.
  The solution of the backward equation is denoted by $(p,q)$. 
 We often write $(p_\tau,q_\tau)=(p(\tau,t,\zeta_t),q(\tau,t,\zeta_t))$, $\tau \in [t,T]$ when we want to emphasize the dependence of $p$ and $q$ on the parameter $\zeta_t$ at time $t$. Analogously, when we want to emphasize the dependence of $\calZ$ on the initial value $\zeta_t$ at time $t$, we write $\calZ_\tau = \calZ(\tau,t,\zeta_t)$.
 
 Define now
    \begin{equation}\label{v_t}
        v(t,z):=p(t,t,z),
    \end{equation} 
    with $p$ the solution to the backward SDE in \eqref{FBSDE}.
\noindent In the sequel we show that $J^g(t,\zeta_t, \hat u )$ in \eqref{optimal tilde J(u)} is such that 
\begin{equation}\label{J=v}
J^g(t,\zeta_t, \hat u ) = v(t,\zeta_t)    
\end{equation}
and that the optimal control $\hat u$ can be retrieved explicitly via a verification theorem once $v(t,\zeta_t)$ is known. 
In order to achieve \eqref{J=v} we proceed as follow. First we study the forward equation $\calZ(\tau,t,\zeta_t)$ in Section \ref{sec: forward}, then we study the backward equation $(p(\tau,t,\zeta_t),q(\tau,t,\zeta_t))$ in Section \ref{sec: backward} and there we prove the crucial identification: 
\begin{equation}\label{identification}
q_t=\nabla_z v(t,\calZ_t)\nu\sigma^g(t,\calZ_t),
\end{equation}
(see Proposition \ref{Masiero 4.4}). In Section \ref{sec: HJB} we provide an approach to find $v(t,z)$ through HJB equations and at last, in Section \ref{sec: optimal}, we prove \eqref{J=v} and we provide a characterization of the optimal control $\hatu$.

Notice that, for \eqref{identification} to hold, the backward process $p(\cdot, t,z)$ has to be differentiable with respect to $z$. This can be obtained by showing that $\calZ_t^u$ is differentiable with respect to the initial condition $\zeta_t$, and by assuming the following:
\begin{hypothesis}\label{HP4.1 Mas}
	Let us assume that
	\begin{itemize}
		\item[1)] There exists $L_1>0$ such that 
		\begin{equation*}
		|\calH(t,z,\xi_1)-\calH(t,z,\xi_2)|\leq L_1 \|\xi_1-\xi_2\|_{Y^{**}}
		\end{equation*}
		for every $\t$, $z \in Y^*$ and $\xi_1,\xi_2\in Y^{**}$.
		\item[2)] For all $\t$,  $\E\left[\int_t^T|\calH(s,0,0)|^2ds\right]<\infty$.
		\item[3)]  For every $\t$ we have $\calH(t,\cdot,\cdot)\in\calG^{1,1}(Y^*\times Y^{**})$.
		\item[4)] There exist $L_2>0$ and $m\geq 0$ such that
		\begin{equation*}
		|\nabla_z\calH(t,z,\xi)h|\leq L_2\|h\|_{Y^*}(1+\|z\|_{Y^*})^m(1+\|\xi\|_{Y^{**}})
		\end{equation*}
		for every $\t$, $z, h\in Y^*$ and $\xi\in Y^{**}$.
		\item[5)] $G^g\in\calG^1(Y^*)$ and there exists $L_3>0$ such that, for every $z_1,z_2 \in Y^*$
		\begin{equation*}
			|G^g(z_1)-G^g(z_2)|\leq L_3 \|z_1-z_2\|_{Y^*}.
		\end{equation*}
	\end{itemize}
\end{hypothesis}
\noindent Further details on the continuous dependence on $\zeta_t$ of the forward equation can be found in Section \ref{sec: forward}, while we refer to Section \ref{sec: backward} for details on the differentiability of $p(\cdot, t,z)$  with respect to $z$.

\begin{obs}
    The identification \eqref{identification}, in the case where $\sigma^g$ does not depend on $\calZ$ in \eqref{FBSDE}, can be proven following \cite{Masiero} and dropping the UMD hypothesis on $Y^*$. In our case though, being $\sigma^g$ dependent on $\calZ_t$, we need to exploit some Malliavin calculus techniques on Banach spaces, and thus assume that $Y^*$ is UMD. 
\end{obs}
 
\subsection{On the lifted forward equation}\label{sec: forward}
In this section we study the lifted forward equation $\calZ_t$ in \eqref{FBSDE}. In particular, we prove that it admits a unique Markovian solution and we study its continuous dependence from the initial parameter $\zeta_t$. 
We thus take
\begin{equation}\label{FORWARD}
	\begin{cases}
	d\calZ_\tau&=\calA^*\calZ_\tau d\tau+\nu\beta^g(\tau,\calZ_\tau)d\tau+\nu\sigma^g(\tau,\calZ_\tau)dW_\tau, \quad \tau\in[t,T],\\
	\calZ_t&=\zeta_t,
	\end{cases}
\end{equation}
we recall that $\calA^*$ is the generator of a uniformly continuous semigroup on the Banach space $Y^*.$
We assume the following:
\begin{hypothesis}\label{HP3.1 FT}
	Suppose that
	\begin{itemize}
		\item[i)] $\beta^g:[0,T]\times Y^*\longrightarrow \R$ is continuous and, for all $\t$ $z_1,z_2\in Y^*$, there exists a constant $L_1 > 0$ such 
 that 
 $$|\nu\beta^g(t,z_1)-\nu\beta^g(t,z_2)|\leq L_1\|z_1-z_2\|_{Y^*},$$
		the map $\nu\beta^g:[0,T]\times Y^*\longrightarrow Y^*$ is measurable. Moreover, for all $\t$ and $z \in Y^*$,
		$$|\nu\beta^g(t,z)|^2\leq L_2(1+\|z\|_{Y^*}^2),$$
        for some constant $L_2>0$.
		\item[ii)] $\sigma^g:[0,T]\times Y^*\longrightarrow \R$ is such that, for every $v\in Y^{**}$ the map $\nu\sigma^g  v:[0,T]\times Y^* \longrightarrow \R $ is measurable, $e^{s\calA^*}\nu\sigma^g(t,z)\in L^2(Y^*)$ for every $s>0$, $t\in[0,T]$ and $z\in Y^*$, and
		\begin{equation*}
			\|\nu\sigma^g(t,z)\|^2_{L(Y^*)}\leq L_3(1+\|z\|_{Y^*}^2),
		\end{equation*}
		for some constant $L_3>0$.
		
		Moreover, for $s>0$, $\t$, $z_1,z_2\in Y^*$ there exists a constant $L_4>0$ such that		\begin{equation*}
			\|\nu\sigma^g(t,\zeta_1)-\nu\sigma^g(t,z_2)\|_{L^2(Y^*)}\leq L_4 \|z_1-z_2\|_{Y^*},
		\end{equation*}

		\item[iii)] For every $s>0$, $\t$, $ \nu\beta^g(t,\cdot)\in\calG^1(Y^*,Y^*)$.
	\end{itemize}
\end{hypothesis}
\noindent Our first result is the following:
\begin{prop}\label{PROP 3.3 FT}
Assume Hypothesis \ref{HP3.1 FT} holds. For every $p\in[2,\infty)$, we have that:
\begin{itemize}
	\item[i)] The map $(t,z)\longmapsto \calZ(\cdot,t,z) $ is in $\calG^{0,1}([0,T]\times Y^*; L^p(\Omega; C([0,T];Y^*)))$.
	\item[ii)] For every $h\in Y^*$ the partial directional derivative process $\nabla_z\calZ(\tau,t,z)h$, $\tau\in[0,T]$ solves $\mathbb{P}$-a.s. the equation
	\begin{align*}
		\nabla_z\calZ(\tau,t,z)h &=e^{(\tau-t)\calA^*}h+\int_t^\tau e^{(\tau-s)\calA^*}\nabla_z\nu\beta^g(s,\calZ(s,t,z))\nabla_z\calZ(s,t,z)h \ ds\\
		&+\int_t^\tau \nabla_z\left(e^{(\tau-s)\calA^*}\nu\sigma^g(s,\calZ(s,t,z)) \right)\nabla_z\calZ(s,t,z)h \ dW_s,\quad \tau\in[t,T],\\
		\nabla_z\calZ(\tau,t,z,)h&=h, \quad \tau\in[0,t).
	\end{align*}
	\item[iii)]$\|\nabla_z\calZ(\tau,t,z)h\|_{L^p(\Omega ; C([0,T];Y^*))}\leq c\|h\|_{Y^*}$ for some constant $c$.
\end{itemize}
We also find  that
\begin{itemize}
	\item[iv)] \eqref{FORWARD} admits a unique adapted solution $\calZ\in L^p(\Omega, C([t,T]);Y^*)$.
\end{itemize} Moreover, we have the following estimate 
	\begin{equation}\label{Stima crescita}
		\|\calZ\|_p^p:=\E\left[\sup_{\tau\in[t,T]}\|\calZ_\tau\|_{Y^*}^p\right]\leq C(1+\|\zeta_t\|_{Y^*}^p),
	\end{equation}
	where $C$ is a constant depending only on $p, T, L$, where $L:=\max\{L_1,L_2, L_3, L_4\}$ and $M:=\sup_{\tau\in[t,T]}\|e^{\tau \calA^*}\|_{L^2(Y^*)}$.
\end{prop}

\begin{proof}
	The proof is inspired by \cite[Theorem 7.4]{DaPrato_Soluzione}  and \cite[Proposition 3.2]{FT}  The main difference with our work are the spaces at play. Consider the map
	\begin{equation*}
		\Phi(\calZ,t,z)_\tau:L^p(\Omega;C([0,T];Y^*))\times[0,T]\times Y^*\longrightarrow L^p(\Omega;C([0,T];Y^*))
	\end{equation*}
	defined as
	\begin{align*}
		\Phi(\calZ,t,z)_\tau&:=e^{\calA^*(\tau-t)}z+\int_0^\tau\mathds{1}_{[t,T]}(s)e^{\calA^*(\tau-s)}\nu\beta^g(s,\calZ_s)ds\\
		&+\int_0^\tau\mathds{1}_{[t,T]}(s)e^{\calA^*(\tau-s)}\nu\sigma^g(s,\calZ_s)dW_s\\
            &:= S_0(\calZ,t,z)_\tau + S_1(\calZ,t,z)_\tau  + S_2(\calZ,t,z)_\tau  \quad \tau \in [0,T].
	\end{align*}
	We want to show that $\Phi$ is a contraction with respect to the first variable. 	
	We notice that 
	\begin{align*}
	    \|S_1(\calZ,t,z)\|^p & \leq M^p\E\left[\left(\int_0^T\|\nu\beta^g(s,\calZ_s)\|_{Y^*} ds \right)^p\right] \\
	    &\leq T^{p-1}M^p\E\int_0^T\E\left[\int_0^T\|\nu\beta^g(s,\calZ_s)\|_{Y^*}^p ds\right]\\
	    &\leq 2^{p/2-1}T^{p-1}M^pL^p\E\left[\int_0^T(1+\|\calZ_s\|_{Y^*}^p)ds\right]\\
	    &\leq 2^{p/2-1}(TLM)^p(1+\|\calZ\|_p),
	\end{align*}
	and
	\begin{align*}
	    \|S_2(\calZ,t,z)\|^p&\leq \sup_{\tau\in[0,T]}\E\left[\left\|\int_0^\tau e^{(\tau-s)\calA^*}\nu\sigma^g(s,\calZ_s)dW(s)\right\|_{Y^*}^p\right]\\
	    &\leq M^p C_{p/2}LT^{p/2-1}2^{p/2-1}\E\left[\int_0^T(1+\|\calZ_s\|^p_{Y^*})ds\right]\\
	    &\leq M^pC_{p/2}L(2T)^{p/2-1}(T+\|\calZ\|_p),
	\end{align*}
	where we used the linear growth conditions on $\beta^g$ and $\sigma^g$ and the Burkholder-Davis-Gundy inequality for $UMD$ Banach spaces (see \cite{BDG}). We thus have showed that $\Phi(\cdot,t,z)$ is a well defined mapping. Now, taking $\calZ_1$ and $\calZ_2$ arbitrary processes in $Y^*$, then
	\begin{align*}
	    \|\Phi(\calZ_1,t,z)-\Phi(\calZ_2,t,z)\|_p&\leq \|S_1(\calZ_1,t,z)-S_1(\calZ_2,t,z)\|_p+\|S_2(\calZ_1,t,z)-S_2(\calZ_2,t,z)\|_p\\
     &:= I_1+I_2.
	\end{align*}
	With computations similar to the ones above, exploiting the Lipschitz condition on $\beta^g$ and $\sigma^g$ (see Hypothesis \ref{HP3.1 FT}, \emph{i)} and \emph{ii)}), one finds that
	\begin{equation*}
	    I_1^p\leq (TML)^p\|\calZ_1-\calZ_2\|_p^p,
	\end{equation*}
	and 
	\begin{equation*}
	    I_2^p\leq C_{p/2}(ML)^pT^{p/2}\|\calZ_1-\calZ_2\|_p^p.
	\end{equation*}
	Summing up, we have that
	\begin{equation*}
	    \|\Phi(\calZ_1)-\Phi(\calZ_2)\|_p\leq LM(T^p+C_{p/2}T^{p/2})^{1/p}\|\calZ_1-\calZ_2\|_p.
	\end{equation*}
	This means that $\Phi(\calZ, t, z)$ is a contraction only for $\t$ when $T$ satisfies
	\begin{equation}\label{condizione_contrazione}
	    LM(T^p+C_{p/2}T^{p/2})^{1/p}<1.
	\end{equation} 
    Condition \eqref{condizione_contrazione} on $T$ can be easily removed by considering the equation on intervals $[0,\tilde{T}]$, $[\tilde{T},2\tilde{T}]$,..., where $\tilde T$ satisfies \eqref{condizione_contrazione}.
	Thanks to the fixed point theorem we find that \eqref{Z_t rewritten} admits a unique solution. We conclude that \eqref{Stima crescita} holds by applying Gronwall's Lemma with arguments in line with \cite[Theorem 7.4 (iii)]{DaPrato_Soluzione}.
	Notice now that, being $\Phi(\cdot,t,z)$ a contraction uniformly with respect to $\t$, $z\in Y^*$, by Proposition \ref{PROP 2.4 FT} we obtain $ii)$ if
	\begin{equation*}
		\Phi\in\calG^{1,0,1}\left(L^p(\Omega;C([0,T];Y^*))\times[0,T]\times Y^*, L^p(\Omega;C([0,T];Y^*)) \right).
	\end{equation*}
	This is verified by a (slight modification) of Lemma \ref{LEMMA 2.3 FT}. Indeed we notice that $\Phi$ is differentiable in $z$. For more details we refer to \cite{DaPrato_Soluzione,FT}.
\end{proof}
\begin{obs}\label{Markovianity Z}
    We notice that $\calZ$ is Markovian (see e.g. \cite[Theorem 1.157]{GF}).
\end{obs}

\begin{co}\label{existence Z controlled}
    Assume Hypothesis \ref{HP5.1 Mas} and \ref{HP3.1 FT} hold. Then \eqref{Z_t controlled} admits a unique solution.
\end{co}
\begin{proof}
    Thanks to the boundedness of $R$, one can apply the Girsanov Theorem, see e.g. \cite[Theorem 10.14]{DaPrato_Soluzione}, and proceed like in the proof of Proposition \eqref{PROP 3.3 FT}.
\end{proof}
 
 \subsection{On the backward equation}\label{sec: backward}
 In this section we study the backward equation
 \begin{equation}\label{p_t} 
	p_\tau=G^g(\calZ_T)+\int_\tau^T\calH(s,\calZ_s,q_s)ds-\int_\tau^Tq_s\nu dW_s,
\end{equation}
introduced in \eqref{FBSDE}. We study existence and uniqueness of a solution as well as its continuous dependence from the parameter $\zeta_t$. Later on we will exploit \eqref{p_t} to prove \eqref{identification}, as well as show that the optimal value $J^g(t,z,\hat u)$ for the optimization problem \eqref{Z_t rewritten} - \eqref{optimal tilde J(u)} is achieved for  $J^g(t,\zeta_t, \hat u) = v(t,\zeta_t) = p(t,t,\zeta_t)$.
 
 We observe that the following a priori estimate for the pair process $(p,q)$ holds (see \cite{Masiero} and \cite[Proposition 4.3]{FT} ):
\begin{equation*}
	\E\Bigg[\sup_{\tau\in[t,T]}|p_\tau|^2\Bigg]+\E\Bigg[\int_t^T\|q_\tau\|^2_{Y^{**}}d\tau\Bigg]\leq c \E\Bigg[\int_t^T|\calH(\tau,0,0)|^2d\tau\Bigg]+ c \E\Bigg[|G^g(\calZ_T)|^2\Bigg],
\end{equation*}
where $c$ is a constant depending on $T$ and $L:=\max \{L_1, L_2, L_3\}$, where $L_i$, $i=1,..,3$ are the coefficients in Hypothesis \ref{HP4.1 Mas}.

\begin{prop}\label{Masiero 4.2}
	Assume that Hypotheses \ref{HP4.1 Mas} and \ref{HP3.1 FT} hold true. Then \eqref{p_t}	admits a unique solution $(p,q)\in L^2(\Omega, C[0,T]; Y^*)\times L^2(\Omega, L^2[0,T];L^2(Y^*))$ such that the map $$z\longmapsto (p(\cdot,\cdot,z),q(\cdot,\cdot,z))\text{ \emph{  belongs to  } }\calG^1(L^\eta(\Omega;C([0,T];Y^*)),\mathcal{K}_{cont}([0,T]))$$ for $\eta=\ell(m+1)(m+2)$, where $\mathcal{K}_{cont}([0,T])$ is the space of adapted processes $(p,q)$ taking values in $\R\times Y^{**}$ such that $p$ has continuous paths and 
	\begin{equation*}
		\E\left[\sup_{\tau\in[0,T]}|p_\tau|^2\right]+\E\left[\int_\tau^T\|q_s\|_{Y^{**}}^2ds\right]<\infty.
	\end{equation*}
	
	\noindent Moreover, for every $\ell\geq 2$.
	\begin{equation*}
	\left(\E\left[\sup_{\t}|\nabla_z p(t,z)h|^\ell\right]\right)^{1/\ell}\leq C\|h\|_{Y^*}\left(1+\|z\|_{Y^*}^{(m+1)^2}\right)
	\end{equation*}
\end{prop}
\begin{proof}
	See \cite{Masiero} Proposition 4.2.
\end{proof}
\noindent Still aiming to prove \eqref{identification}, we provide yet another crucial result that links the directional derivative of $\calZ$ to its Malliavin derivative.
\begin{prop}\label{3.8 FT}
	Assume that Hypothesis \ref{HP3.1 FT} holds. Then for almost all $s,\tau$ such that
	$t\leq s \leq \tau < T$ we have that
	\begin{equation}\label{FT 3.24}
		D_s\calZ(\tau,t,z)=\nabla_z \calZ(\tau,s,\calZ(s,t,z))\nu\sigma^g(s,\calZ(s,t,z)), \ \mathbb{P}-a.s.
	\end{equation} 
	moreover 
 \begin{equation}\label{FT 3.24 ii}
     D_s\calZ(T,t,z)=\nabla_z \calZ(T,s,\calZ(s,t,z))\nu\sigma^g(s,\calZ(s,t,z)), \ \mathbb{P}-a.s. \text{ for almost all }s.
 \end{equation}
\end{prop}
\begin{proof}
	Thanks to Proposition \ref{PROP 3.3 FT}, for every $s\in[0,T]$ and every direction $h\in Y^*$, the directional derivative process $\nabla_z \calZ(\tau,s,z)h$, $\tau\in[s,T)$ solves $\mathbb{P}$-a.s. the equation
	\begin{align*}
		\nabla_z\calZ(\tau,t,z)h &=e^{(\tau-t)\calA^*}h+\int_t^\tau e^{(\tau-s)\calA^*}\nabla_z\nu\beta^g(s,\calZ(s,t,z))\nabla_z\calZ(s,t,z)hds\\
		&+\int_t^\tau \nabla_z\left(e^{(\tau-s)\calA^*}\nu\sigma^g(s,\calZ(s,t,z)) \right)\nabla_z\calZ(s,t,z)hdW_s,\quad \tau\in[t,T],\\
		\nabla_z\calZ(\tau,t,z)h&=h, \quad \tau\in[0,t),
	\end{align*}
	Given $v\in Y^*$ and $t\in[0,s]$, we can replace $z$ by $\calZ(s,t,z)$ and $h$ by $\nu\sigma^g(s,\calZ(s,t,z))v$ in the previous equation, since $\calZ(s,t,z)$ is $\F_s$ measurable. Note now that
	\begin{equation*}
		\calZ(\eta,s,\calZ(s,t,z))=\calZ(\eta,t,z) \quad \mathbb{P}-a.s.,
	\end{equation*}
	for $\eta\in[s,T)$, as a consequence of the uniqueness of the solution of \eqref{FORWARD}. This yields
	\begin{align*}
		\nabla_z&\calZ(\tau,\calZ(s,t,z))\nu\sigma^g(s,\calZ(s,t,z))v=e^{(\tau-s)\calA^*}\nu\sigma^g(s,\calZ(s,t,z))v\\
		&+\int_s^\tau e^{(\tau-\eta)\calA^*}\nabla_z\nu\beta^g(\eta,\calZ(\eta,t,z))\nabla_z\calZ(\eta,s,\calZ(s,t,z))\nu\sigma^g(s,\calZ(s,t,z))vd\eta\\
		&+\int_s^\tau\nabla_z(e^{(\tau-\eta)\calA^*}\nu\sigma^g
		(\eta,\calZ(\eta,t,z)))\nabla_z\calZ(\eta,s,\calZ(s,t,z))\nu\sigma^g(s,\calZ(s,t,z))vdW_\eta,
	\end{align*}
	for $\tau\in[s,T)$, $\mathbb{P}$-a.s. This shows that the process $$\left\{\nabla_z\calZ(\tau,t,\calZ(s,t,z,))\nu\sigma^g(s,\calZ(s,t,z))v \right\}_{t\leq s \leq \tau < T},$$ is a solution of the equation
	\begin{align*}
		Q_{s,\tau}&=e^{(\tau-s)\calA^*}\nu\sigma^g(s,\calZ_s)v+\int_s^\tau e^{(\tau-\eta)\calA^*}\nabla_z\nu\beta^g(\eta,\calZ_\eta)Q_{s,\eta}d\eta\\
		&+\int_s^\tau \nabla_z(e^{(\tau-\eta)\calA^*}\nu\sigma^g(\eta,\calZ_\eta))Q_{s,\eta}dW_\eta,
	\end{align*}
	where $Q_{s,\tau}:=D_s\calZ_\tau v$. The thesis now follows from the uniqueness property, as proved e.g. in \cite{FT} Proposition 3.5.
	To complete the proof of \eqref{FT 3.24 ii}, we take a sequence $\tau_n\uparrow T$ such that \eqref{FT 3.24} holds for every $\tau_n$, and we let $n\rightarrow\infty$ (see \cite{FT}). The result follows from the regularity properties of $D\calZ$ and $\nabla_z\calZ$, as well as the closedness of the operator $D$ on UMD Banach spaces.
\end{proof}

\noindent
In this framework, using the results presented in \cite{Clark-Ocone} and \cite{MalliavinBanach} we find that:
\begin{prop}\label{5.6 FT}
	Assume Hypotheses \ref{HP4.1 Mas} - \ref{HP3.1 FT} Then for a.a. $s,\tau$ such that $t\leq s\leq \tau \leq T$ we have that
	\begin{align}
		D_sp(\tau,t,z)&=\nabla_z p(\tau,s,\calZ(s,t,z))\nu\sigma^g(s,\calZ(s,t,z)) \quad \mathbb{P}-a.s.,\\
		D_sq(\tau,t,z)&=\nabla_z q(\tau,s,\calZ(s,t,z))\nu\sigma^g(s,\calZ(s,t,z)) \quad \mathbb{P}-a.s..
	\end{align}
	Moreover, for a.a. $s\in[t,T]$,
	\begin{equation}\label{q(s,t,z)}
		q(s,t,z)=\nabla_z p(s, s,\calZ(s,t,z))\nu\sigma^g(s,\calZ(s,t,z)) \ \mathbb{P}-a.s.
	\end{equation}
\end{prop}
\begin{proof}
	The proof follows the same arguments as \cite[Proposition 5.6]{FT} though the spaces at play are different. Indeed, the main tools are provided in Proposition \ref{3.8 FT}. So, thanks to the extension of Malliavin calculus to UMD Banach spaces, and the chain rule linking Malliavin derivative and Gateaux derivative (see Proposition \ref{ChainRule}), the result is secured. \end{proof}
\noindent Finally, the next result provides the proof of \eqref{identification}.

\begin{prop}\label{Masiero 4.4}
	    Assume that Hypothesis \ref{HP4.1 Mas} and \ref{HP3.1 FT} hold true. Then the function $v(t,z):=p(t,t,z)$ in \eqref{v_t} is continuous and for every $\t$, $v(t,\cdot)$ belongs to $\calG^1(Y^*,\R)$ and there exists $C>0$ such that 
	\begin{equation*}
	|\nabla_z v(t,z)h|\leq C\|h\|_{Y^*}(1+\|z\|_{Y^*}^{(m+1)^2}).
	\end{equation*} 
	Moreover we have that 
    \begin{equation*}
        q(s,t,z)=\nabla_z v(t,\calZ(s,t,z))\nu\sigma^g(s,\calZ(s,t,z)).
    \end{equation*}
\end{prop}
\begin{proof}
    The first part is a corollary of Proposition \ref{Masiero 4.2}. The second is derived from \eqref{q(s,t,z)}.
\end{proof}

\subsection{The HJB equation}\label{sec: HJB}

Formally define 
\begin{align*}
	\mathcal L_t[f](x)&:=\frac 12 \text{Trace}(G^g(t,x)G^g(t,x)^* \ \nabla^2f(x))+\langle \calA^*x+\nu\beta^g(t,x),\nabla f(x) \rangle_{Y^*\times Y^{**}}.
\end{align*}
\noindent
We can consider the Hamilton-Jacobi-Bellman equation associated with the control problem \eqref{Z_t rewritten} - \eqref{optimal tilde J(u)}, which is given by
\begin{equation}\label{HJB}
	\begin{cases}
	\frac{\partial w}{\partial t}(t,z)&=-\mathcal{L}_t w(t,z)-\calH(t,z,\nabla w(t,z) \nu\sigma^g(t,z)),\\
	w(T,z)&=G^g(z).
	\end{cases}
\end{equation}
\noindent A solution of this equation provides a way to compute $v(t,z)$ in \eqref{v_t} by PDE methods (see e.g. \cite{cannarsadaprato}). The connection between \eqref{HJB} and \eqref{v_t} is actually detailed in the forthcoming Theorem \ref{HJB Teo} by means of the forward backward system \eqref{FBSDE}. Later on, in Theorem \ref{Masiero 5.7}, we shall see how $w(t,z)$ is connected with the optimal performance, see \eqref{J=v}.
Thus, we are interested in finding mild solutions to the previous equation, which we are going to defined soon. This problem was tackled in \cite{FT} (for a Hilbert space) in the case of a general $\nu\sigma^g$, and in \cite{Masiero} (for a Banach space) in the case of a constant $\nu\sigma^g$. Our result is then extending the on of \cite{Masiero}.

Let $\calZ(\tau,t,\zeta_t)$ be a solution to \eqref{FORWARD}, with $\calA^*$, $\beta^g$ and $\sigma^g$ satisfying Hypotheses \ref{HP4.1 Mas} - \ref{HP3.1 FT}. We recall that this solution is a $Y^*$-valued Markov process (see Remark \ref{Markovianity Z}). We can then define the transition semigroup on continuous and bounded functions $\varphi:Y^*\longrightarrow\R$ as
\begin{equation*}
	P_{t,\tau}[\varphi](z)=\E\left[\varphi(\calZ(\tau,t,z))\right].
\end{equation*}
Moreover, we have that this semigroup is also well defined on continuous functions $\varphi:Y^*\longrightarrow\R$ with polynomial growth with respect to $z$. 

\begin{defin}
	A function $w:[0,T]\times Y^*\longrightarrow \R$ is a mild solution of the Hamilton-Jacobi-Bellman equation \eqref{HJB} if:
	\begin{itemize}
		\item For every $\t$ $w(t,\cdot)\in \calG^1(Y^*)$, $w$ is continuous and $(t,z)\longmapsto w(t,z)$ is measurable from $[0,T]\times Y^*$ with values in $Y^{**}$
		\item For every $\t$, there exists $C>0$ such that $|w(t,z)|\leq C(1+\|z\|^j_{Y^*})$ and $|\nabla_z w(t,z)h|\leq C\|h\|_{Y^*}(1+\|z\|^k_{Y^*})$, with $z,h\in Y^*$ and $j$ and $k$ positive integers.
		\item The following equality holds. 
		\begin{equation*}\label{6.2 MASIERO}
	w(t,z)=P_{t,T}[G^g](z)+\int_t^TP_{t,\tau}[\calH(\tau,\cdot,\nabla w (\tau,\cdot)\nu\sigma^g(t,\cdot))](z)d\tau, \quad \t, z\in Y^*.
\end{equation*}
	\end{itemize}	
\end{defin}
\noindent In order to prove that there exists a unique solution of \eqref{HJB} we need once again the forward-backward system \eqref{FBSDE}:
\begin{equation*}
	\begin{cases}
	d\calZ_\tau&=\calA^*\calZ_\tau d\tau+ \nu\beta^g(\tau,\calZ_\tau)+\nu\sigma^g(\tau,\calZ_\tau)dW_\tau,\qquad \tau\in[t,T]\\
	\calZ_t&=\zeta_t\\
	dp_t&=-\calH(\tau,\calZ_\tau,q_\tau)d\tau+q_\tau \nu dW_\tau,\qquad \tau\in[t,T]\\
	p_T&=G^g(\calZ_T)
	\end{cases}
\end{equation*}

\begin{teo}\label{HJB Teo}
	Assume that $G^g$ and $\calH$ satisfy Hypothesis \ref{HP4.1 Mas} and that Hypothesis \ref{HP3.1 FT} hold true. Then there exists a unique mild solution of the Hamilton-Jacobi-Bellman equation \eqref{HJB} given by
	\begin{equation}
		w(t,\zeta_t)=v(t,\zeta_t)
	\end{equation}
	where $(\calZ,p,q)$ is the solution of  \eqref{FBSDE} and $v(t,z) = p(t,t,z)$, see \eqref{v_t}.
\end{teo}
\begin{proof}
	The proof is based on arguments similar to \cite[Theorem 6.2]{Masiero} and \cite[Theorem 6.2]{FT}, though adapted to the current framework. Note that the main difference with \cite{FT} is the nature of the spaces considered.
\end{proof}

\subsection{Solving the optimal control problem}\label{sec: optimal}
As we have proven the identification \eqref{identification} to be true (see Proposition \ref{Masiero 4.4}), we can finally move to the study of the optimal control problem \eqref{Z_t rewritten} - \eqref{optimal tilde J(u)}. As anticipated, we want to show that the optimal value 
$$\inf_{u\in\mathds A}J^g(t,\zeta_t, u) = J^g(t,\zeta_t, \hat u) = v(t,\zeta_t),$$
where we have defined $v(t,\zeta_t) = p(t,t,\zeta_t)$ in \eqref{v_t}, $(p,q)$ solve the backward stochastic differential equation \eqref{FBSDE} and a solution of $v(t,\zeta_t)$ can be obtained through the HJB equation \eqref{HJB} (see Theorem \ref{HJB Teo}).
We define the, possibly empty, set 
\begin{equation}\label{Gamma}
\Gamma(\tau,z,\xi)=\left\{u\in\calU: F^g(\tau,z,u)+\xi \nu R^g(\tau,z,u)=\calH(\tau,z,\xi) \right\},
\end{equation}
where $\tau\in[t, T]$, $z\in Y^*$, $\xi \in Y^{**}$.

\begin{hypothesis}\label{Gamma non vuoto}
    We notice that, intuitively, $\Gamma(t,z,\xi)$ represents the set of cotrols that allow us to obtain the minimum in the Hamiltonian \eqref{Hamilton}. We will thus assume that for all $\tau\in[t, T]$, $z\in Y^*$, $\xi \in Y^{**}$, $\Gamma(\tau, z,\xi) \neq \emptyset$
\end{hypothesis}

\begin{obs}\label{Gamma0}
    Thanks to the Filippov theorem (see \cite{Filippov}), being $\Gamma(\tau, z,\xi)$ non empty for all $\tau \in [t,T]$, $z\in Y$, $\xi \in Y^{**}$, there exists a Borel measurable map $\Gamma_0:[0,T]\times Y^* \times Y^{**}\longrightarrow \calU$ such that, for $\t$, $z\in Y^* $ and $\xi\in Y^{**}$, $\Gamma_0(t,z,\xi)\in\Gamma(t,z,\xi)$. 
\end{obs}

\begin{prop}
	Assume that Hypothesis \ref{HP5.1 Mas} - \ref{HP4.1 Mas} and \ref{HP3.1 FT} and hold true, and let $v$ be defined in \eqref{v_t} and $u$ in $\mathds A$. Then for all $\t$ and $z \in Y^*$, we have that
	\begin{equation*}
		J^g(t,z,u)\geq v(t,z).
	\end{equation*}
\end{prop}
\begin{proof}
%	The proof of this result follows the one in \cite[Proposition 5.5]{Masiero}. We report it here to highlight the main differences with our work. 	
	Let $u$ in $\mathds A$ and take $\calZ_\tau^u$ to be a solution of \eqref{Z_t controlled}  corresponding to the control $u$ (for the existence of such a solution see Corollary \ref{existence Z controlled}). Define
	\begin{equation*}
		W_\tau^u =W_\tau +\int_{t\vee \tau}^\tau R^g(s,\calZ_s^u,u_s)ds, \quad \tau\in[0,T].
	\end{equation*}
	We notice that $\calZ_\tau^u$ solves the equation 
	\begin{equation*}
			\begin{cases}
			d\calZ^u_\tau=&\calA^*\calZ_\tau^ud\tau+\nu\beta^g(\tau,\calZ^u_\tau)d\tau+\nu\sigma^g(\tau,\calZ^u_\tau)dW^u_\tau,\\
			\calZ_t=&\zeta_t,
		\end{cases}
	\end{equation*}
	and, being $R$ and thus $R^g$ bounded, we can find a probability $\mathbb{P}^u$ equivalent to $\mathbb P$ such that $W^u$ is a Wiener process on $\R$ and thus $\nu W^u$ is a cylindrical Wiener process with values in $Y^*$ (see \cite[Theorem 7.2 (iii)]{DaPrato_Soluzione}).
	We consider the backward equation  w.r.t. $\mathbb{P}^u$ for the unknowns $(p_\tau^u,q_\tau^u)$, $\tau\in[t,T]$ given by
	\begin{equation}\label{Masiero 5.6}
		p_\tau^u+\int_\tau^T q_s^u \nu dW_s^u = G^g(\calZ_T^u)+\int_\tau^T\calH(s,\calZ_s^u,q_s^u)ds.
	\end{equation}
	By taking $(p^u,q^u)$ at $\tau=t$ in \eqref{Masiero 5.6}, we get that $p(t,t,\zeta_t)$ depends only on $t,\zeta_t,\beta^g,\sigma^g, G^g$ and $\calH$. With the same approach as in \cite[Proposition 5.5]{Masiero}, one obtains immediately that $\int_\tau^T q_s^u\nu dW^u_s$ is actually a $\mathbb{P}^u$-martingale.

	By recalling that $v(t,z) := p(t,t,z)$ (see \eqref{v_t}) and taking expectation with respect to the original probability $\mathbb{P}$, we obtain that
	\begin{equation*}
		v(t,\zeta_t)=\E[G^g(\calZ_T^u)]+\E\left[\int_t^T\calH(s,\calZ_s^u,q_s^u)-q_s^u \nu R(s,\calZ_s^u,u_s)  ds\right].
	\end{equation*}
	Adding and subtracting $\E\left[\int_t^T F^g(s,\calZ_s^u,u_s)ds\right]$, we arrive at
	\begin{equation}\label{Masiero 5.8}
		v(t,\zeta_t)=J^g(t,\zeta_t,u)+\E\left[\int_t^T\calH(s,\calZ_s^u,q_s^u)-q_s^u\nu R^g(s,\calZ_s^u,u_s)-F^g(s,\calZ_s^u,u_s)ds\right].
	\end{equation}
	Noticing that
	\begin{equation*}
		\calH(s,\calZ_s^u,q_s^u)-q_s^u\nu R^g(s,\calZ_s^u,u_s)-F^g(s,\calZ_s^u,u_s)\leq 0,
	\end{equation*}
 by definition of $\calH$ in \eqref{Hamilton}, we conclude.
\end{proof}
\begin{co}\label{Corollary 5.6 Masiero}
	Let $\t$ and $\zeta_t \in Y^*$. If $J^g(t,\zeta_t,u^*)=v(t,\zeta_t)$ then $u^*$ is optimal for the control problem starting from $\zeta_t$ at time $t$. Assume Hypothesis \ref{Gamma non vuoto} holds and take $\Gamma_0(\tau, z ,\xi)$ be the Borel measurable map defined in Remark \ref{Gamma0}. Then, an admissible control satisfying 
	\begin{equation*}
		\hat u_\tau=\Gamma_0(\tau, \calZ_\tau^{\hat u},q_\tau^{\hat u }), \quad \mathbb{P}-a.s. \text{ for a.e. }\tau\in[t,T]
	\end{equation*}
	is optimal and $J^g(t,z,\hatu)=v(t,z)$.
\end{co}
\begin{proof}
    The proof is in line with the one in \cite[Corollary 5.6]{Masiero}.
\end{proof}

\begin{teo}\label{Masiero 5.7}
	Assume that Hypothesis \ref{HP5.1 Mas}, \ref{HP4.1 Mas}, \ref{HP3.1 FT} and \ref{Gamma non vuoto} hold true. For all admissible controls $u$ in $\mathds A$, we have that
	\begin{equation*}
	{J^g}(t,\zeta_t,u)\geq v(t,\zeta_t),
	\end{equation*}
	and the equality holds true if and only if 
	\begin{equation}\label{5.12 Mas}
	\hat u_\tau\in\Gamma(\tau,\calZ^{u}_\tau,\nabla v(\tau,\calZ_\tau^u)\nu\sigma^g(\tau,\calZ^u_\tau)) \quad \mathbb{P}-a.s. \text{ for a.a. } \tau\in[t,T].
	\end{equation} 
	Moreover, let us denote by $\Gamma_0(\tau,z,\xi)$ be the measurable selection of $\Gamma(t,z,\xi)$ defined in Remark \ref{Gamma0}. A control satisfying the \emph{feedback law}, defined as:
	\begin{equation}\label{5.13 Mas}
	u_\tau=\Gamma_0(\tau,\calZ^{u}_\tau,\nabla v(\tau,\calZ^u_\tau)\nu\sigma^g(\tau,\calZ_\tau^u)) \quad \mathbb{P}-a.s. \text{ for a.a. } \ \tau\in[t,T]
	\end{equation}
	is optimal.
	Define the \emph{closed loop equation}:
	\begin{equation}\label{5.14 Mas}
		\begin{cases}
		\tZ_\tau=&\left[\calA^*\tZ_\tau+\nu\beta^g(\tau,\tZ_\tau)+\nu\sigma^g(\tau,\tZ_\tau)R^g\Bigg(\tau,\tZ_\tau,\Gamma_0\Big(\nabla v(\tau,\tZ_\tau)\nu\sigma^g(\tau,\tZ_\tau)\Big)\Bigg)\right]d\tau\\
		&+\nu \sigma^g(\tau,\tZ_\tau)dW_\tau, \ \ \tau\in[t,T]\\
		\tZ_t=&\zeta_t.
		\end{cases}
	\end{equation}
	Then \eqref{5.14 Mas} admits a weak solution which is unique in law, and the corresponding pair $(u,\tZ^u)$ is optimal.
\end{teo}
\noindent For more details about the definition of \emph{feedback law} and \emph{closed loop equation} in the case of optimal control for Hilbert spaces we refer to \cite[Section 2.5]{GF}.
\begin{proof}
Using \eqref{Masiero 5.8} and Proposition \ref{5.6 FT} we can rewrite $v(t,\zeta_t)$ as 
	\begin{align*}
		v(t,\zeta_t)=&J^g(t,\zeta_t,u)+\int_t^T\left[\calH(\tau,\calZ^u_\tau,\nabla v(\tau,\calZ^u_\tau)\nu\sigma^g(\tau,\calZ^u_\tau))\right.\\
		&\left. -\nabla v(\tau,\calZ^u_\tau)\nu\sigma^g(\tau,\calZ^u_\tau)) R^g(\tau,\calZ_t^u,u_\tau)-F^g(\tau,\calZ^u_\tau,u_\tau)\right]d\tau.
	\end{align*} 
	The proof of the first statement now follows from Corollary \ref{Corollary 5.6 Masiero}.
	The closed loop equation can be solved in the weak sense via a Girsanov change of measure. 
	Recall that $(\Omega, \mathcal F, \mathbb P)$ is the probability space on which the Wiener process $(W_\tau)_{\tau\geq 0}$ in \eqref{FBSDE} is defined. Define $(\hat{W}_\tau)_{\tau\geq 0}$ as
	\begin{equation*}
		\hat W_t:= W_t-\int_0^t R^g\Bigg(\tau,\tilde\calZ_\tau^u,\Gamma_0\Big(\tau,\tilde{\calZ}_\tau,\nabla v(s,\tilde{\calZ}_\tau)\nu\sigma^g(\tau,\tilde{\calZ}_\tau) \Big) \Bigg)d\tau.
	\end{equation*}
	Due to the Girsanov theorem there exists a probability $\hat{\mathbb{P}}$ on $\Omega$ such taht $\hat W_\tau$ is a Wiener process. We then notice that $\nu W$ and $\nu \hat{W}$ are cylindrical Wiener processes with values in $Y^*$, and that the closed loop equation \eqref{5.14 Mas} can be rewritten under $\hat{\mathbb{P}}$ as 
	\begin{equation*}
		\begin{cases}
			d\tilde\calZ^u_\tau=&\calA^*\tilde\calZ_\tau^u d\tau+\nu\beta^g(\tau,\tilde\calZ^u_\tau)d\tau+\nu\sigma^g(\tau,\tilde\calZ^u_\tau)d\hat W_\tau,\\
			\calZ_t=&\zeta_t.
		\end{cases}
	\end{equation*}
	Then, thanks to Proposition \ref{PROP 3.3 FT}, we have a unique solution to this new process related to the probability $\hat{\mathbb{P}}$ and the Wiener process $\hat{W}_\tau$, which implies that also the closed loop equation \eqref{5.14 Mas} always admits a solution in the weak sense.
    Thanks to Hypothesis \ref{Gamma non vuoto}, we know that $\Gamma(\tau, \calZ_\tau^u, \nabla v(\tau, \calZ_\tau^u)\nu\sigma^g(\tau, \calZ_\tau^u))$ is non empty and thus, by the Filippov theorem, a measurable selection $\Gamma_0(\tau, \calZ_\tau^u, \nabla v(\tau, \calZ_\tau^u)\nu\sigma^g(\tau, \calZ_\tau^u))$ of $\Gamma$ exists and the optimal control can be obtained. This proof is in line with \cite[ Theorem 5.7]{Masiero} and \cite[Theorem 7.2]{FT}. 
    \end{proof}

\begin{obs}
	Notice that, having solved the lifted optimization problem, thanks to \eqref{lift}-\eqref{equivalence J(u)}, we have also solved the original problem \eqref{J(u)}-\eqref{X_t}. Indeed, we have that a control $\hat u$ which is optimal for \eqref{optimal tilde J(u)} where the dynamics for the forward process are given by $\calZ^u_\tau$ in \eqref{Z_t controlled} is also optimal for the original problem \eqref{J(u)}, as $J^g(t,\zeta_t,u)=J(t,x,u)$ by definition and $ X^u(t)=\langle g, \calZ_t^u\rangle$.
\end{obs}

\section{A problem of optimal consumption}
    A cash flow admits consumption with rate $c$ according to the forward dynamics
	\begin{equation}\label{cashflow}
		X^c_t=x(t)+\int_0^tK(t-\tau)\ \mu(\tau, X^c_\tau) d\tau + \int_0^t K(t-\tau)\ \sigma(\tau,X_\tau^c)(-R(c_\tau) d\tau+dW_\tau),
	\end{equation}
	where $x:[0,T]\longrightarrow \R$, $K(t-\tau) = \sqrt{t-\tau}$, $\mu$ and $\sigma$ satisfy Hypothesis \ref{HP3.1 FT}, and $R$ satisfy Hypothesis \ref{HP5.1 Mas}. 
	In this case we lift $K(t):=\sqrt{t}$ on $L^q(\R)\times L^p(\R)$ for $q\in(1,2)$ and $p$ such that $\frac 1 p + \frac 1 q =1$ by considering
	$$e^{\calA^* t}:L^p(\R)\longrightarrow L^p(\R)$$
	the left shift semigroup defined as $(e^{\calA^* t} f)(s)=f(s-t)$, for all $f\in L^p(\R)$, $g(x)=\frac{1}{2\sqrt{x}}\mathds{1}_{[0,1]}(x)$, $\nu(x)=\mathds{1}_{[-1,0]}(x)$.
        Then we have that
        \begin{equation*}
	       K(t)=\langle g, e^{\calA^* t}\nu\rangle_{L^p(\R)\times L^q(\R)}=\int_\R \frac{1}{2\sqrt{x}}\mathds{1}_{[0,1]}(x)\mathds{1}_{[-1,0]}(x-t) dx =\int_0^t \frac{dx}{2\sqrt{x}}=\sqrt{t}.
	    \end{equation*}	
\noindent	We consider a classical optimal control problem given by the maximization of the performance functional
	\begin{equation}\label{functional J example}
		J(t,x,c)=\E\left[\int_t^TF(\tau,X_\tau^c,c_\tau)d\tau+G(X_T^c)\right],
	\end{equation}
	for some functions $F(\tau, X_\tau^c, c_\tau) = F(c_\tau):= -a_1 c_\tau^2$ and $G(X_t^c):= a_2 X_T^c$ ($a_1, a_2 \in \R_{>0}$) satisfying Hypothesis \ref{HP4.1 Mas}. Linear-quadratic performance functionals such as \eqref{functional J example} appear, for example, when considering optimal advertising problems (see e.g. \cite{GM2006, GMS2009} and \cite{Antony}). In this case we have that the stochastic control problem can be reformulated in $Y^*$ with forward dynamics given by

	\begin{equation}
		\calZ_\tau^c=\zeta_0+\int_0^\tau \calA^*\calZ_s^c+\nu\mu^g(s,\calZ_s^c)ds+\nu\sigma^g(s,\calZ_s^c)(-R(c_s) ds+dW_s),
	\end{equation}
	with $\zeta_t = e^{\calA^*t}\zeta $ such that $x(t) = \langle g,\zeta_t \rangle$ for all $\t$. The goal is to minimize
  	\begin{equation}\label{functional J example (lift)}
  	J^g(0,\zeta_0,c)=\E\left[\int_0^T-a_1 c_\tau^2 d\tau+a_2\langle g, \calZ^{c}_T\rangle\right].
  	\end{equation}
  	In this case the Hamiltonian functional \eqref{Hamilton} is given by
  	\begin{equation}\label{Hamilton example}
  		\calH(t,\calZ,\xi)=\inf_{c\in\calU}[-a_1 c^2-\xi \nu R (c)]
   	\end{equation}
	and the forward-backward system is 
	\begin{equation}\label{FBSDE example}
		\begin{cases}
		d\calZ_\tau&=\calA^*\calZ_\tau d\tau+\nu\mu^g(s,\calZ_\tau)d\tau+\nu\sigma^g(s,\calZ_s) dW_\tau, \quad \tau\in[0,T],\\
		\calZ_0&=\zeta_0,\\
		dp_\tau&=-\calH(\tau,\calZ_\tau^c,q_\tau)d\tau+q_\tau \nu dW_\tau,\quad \tau\in[t,T],\\
		p_T&=a_2\langle g, \calZ_T \rangle.
		\end{cases}
	\end{equation}
	In particular, using \eqref{Hamilton example}, we have that
	\begin{equation*}
	p_t=a_2\langle g, \calZ_T \rangle-\int_t^T\inf_{c\in\calU}\left(- a_1 c^2-q_s\nu R(c)\right)ds+q_s \nu dW_s \ s\in[t,T].
	\end{equation*}
	We thus get that the set $\Gamma$ defined in \eqref{Gamma} is
	\begin{equation}\label{Gamma example}
		\Gamma (t,\calZ,\xi)=\Big\{c\in\calU : -a_1 c^2-\xi\nu R(c)=\calH(t,\calZ,c) \Big\},
	\end{equation}
	and thus the optimal $u_\tau$ can be characterized by Theorem \ref{Masiero 5.7} as 
	\begin{equation}\label{Equality example}
		c_\tau=\Gamma_0(\tau,\calZ^{c}_\tau,\nabla v(\tau,\calZ_\tau^c)\nu\sigma^g(\tau,\calZ_\tau^c)) \quad  \ \mathbb{P}-a.s. \text{ for a.a. } \tau\in[0,T],
	\end{equation} 
	for a certain function $\Gamma_0$ such that $\Gamma_0(t,\calZ,\xi)\in \Gamma(t,\calZ,\xi)$.
	In this case the HJB equations \eqref{HJB} become
	\begin{equation}\label{HJB example}
		\begin{cases}
		\frac{\partial v}{\partial t}(t,\zeta_0)&=-\mathcal{L}_tv(t,\calZ(t,0,\zeta_0))\\&\quad-\inf_{c\in\calU}[F^g(t,\calZ(t,0,\zeta_0),c_t)-\nabla v(t,\calZ(t,0,\zeta_0)) \nu\sigma(t,\calZ(t,0,\zeta_0)) c_t],\\
		v(T,\zeta_0)&=a_2 X(0).
		\end{cases}
	\end{equation}
	Where, we remind, $\calZ(\tau,0,\zeta_0)=\calZ_\tau$, $\tau\in[t,T]$, $\calZ(0)=\zeta_0$ and 
	\begin{align*}
		\mathcal L_t[v](t,t,z)&:=\frac 12 \text{Trace}(G^g(v(t,z)G^g(v(t,z))^* \ \nabla^2v(t,t,z))\\&\quad +\langle \calA^*v(t,t,z)+\nu\beta^g(t,v(t,t,z)),\nabla v(t,t,z) \rangle_{L^q([0,\infty))\times L^p([0,\infty))}.
	\end{align*}
 For more details about solving the HJB equation \eqref{HJB example} we refer to \cite{Antony}. Now, thanks to Theorem \ref{HJB Teo} we have that
	\begin{teo}
		Equation \eqref{HJB example} has a unique mild solution $v$. If the cost is given by \eqref{functional J example (lift)}, then for all admissible couples $(c,z)$ we have that $J(t,z,c)\geq v(t,z)$, and the equality holds if and only if 
$\hatc \in \Gamma(t,\calZ,\xi)$ in \eqref{Gamma example}, characterized as \eqref{Equality example}.
Vice versa, if \eqref{Equality example} holds, then 
 \begin{align*}
     \tilde \calZ_\tau &= \calA^*\tilde \calZ_\tau +\nu\mu^g(\tau,\calZ_\tau^c)d\tau\\
     &\quad +\nu\sigma^g(\tau,\calZ_\tau^c)\Big(-R(\Gamma_0(\tau,\tilde\calZ^{c}_\tau,\nabla v(\tau,\tilde\calZ_\tau^c)\nu\sigma^g(\tau,\tilde\calZ_\tau^c))) d\tau+dW_\tau\Big), \quad \tau\in[0,T],
 \end{align*} 
 with initial condition $\tilde\calZ_0 = \zeta_0$, admits a weak solution, which is unique in law, and the corresponding pair $(c,\tilde \calZ^c)$ is optimal.
	\end{teo}
	\begin{obs}
		This gives us a characterization of the optimal control for the lifted problem \eqref{functional J example (lift)} and thus also for \eqref{functional J example}: in fact in our case we have that $X_t=\langle 1, \calZ_t \rangle$. Thanks to this we are able to find the optimal process $X^{\hatu}$. Moreover, we note that, the optimal control for the lifted problem and the optimal control for the original problem coincide., This allows us to retrieve the optimal pair $(\hatu,X^{\hatu})$ for the optimal control problem \eqref{functional J example}. Lastly we notice that the HJB equations \eqref{HJB example}, gives us the optimal value $J^g(0,\zeta_0,\hatu)$, which in turn gives us the optimal value of $J(t,x,\hatu)$, thanks to \eqref{equivalence J(u)}.
	\end{obs}
	
	\begin{obs}
        Inspired by \cite{fiacco}, we could also have considered the kernel $K(t)=\frac{1}{t+\varepsilon}$, $\varepsilon>0$ in \eqref{cashflow}. In this case we can take the space of $L^p([0,\infty))$ of $L^p$ functions on $[0,\infty)$ and its dual $L^q([0,\infty))$ of measures with density in $L^q$, where $\frac 1 p + \frac 1 q =1 $ and $p> 1$. We have that 
	\begin{equation*}
	    K(t)=\frac{1}{t+\varepsilon}=\langle g, \calS_t^* \nu\rangle_{L^p([0,\infty))\times L^q([0,\infty))},
	\end{equation*}
    where $g=\nu=e^{-x\varepsilon/2}$, and $\calS_t^*=e^{-tx}$ i.e. $K(t)$ is the Laplace trasnform of $e^{-x\varepsilon}$. We notice that kernel is liftable (see Definition \ref{liftable}) and that we are in a UMD Banach space.
    It is clear that $e^{-\varepsilon x/2}$ is actually in $L^p([0,\infty))$ for all $p\geq 1$. In particular we can also  take $p=q=2$ and work on the Hilbert space $L^2([0,\infty))$.
	\end{obs}

	\subsection*{Acknowledgment} We would like to thank Anton Yurchenko-Tytarenko and Dennis Schroers for the nice input on kernel decompositions. The research leading to these results is within the project STORM: Stochastics for Time-Space Risk Models, receiving founding from the Research Council of Norway (RCN). Project number: 274410.

\bibliographystyle{plain}
\bibliography{bib}

\begin{thebibliography}{10}

\bibitem{Pham}
E.~Abi~Jaber, E.~Miller, and H.~Pham.
\newblock Linear-{Q}uadratic control for a class of stochastic {V}olterra
  equations: solvability and approximation.
\newblock {\em The annals of probability}, 31:2244--2274, 2021.

\bibitem{Oksendal2}
N.~Agram and B.~Øksendal.
\newblock Malliavin {C}alculus and {O}ptimal {C}ontrol of {S}tochastic
  {V}olterra {E}quations.
\newblock {\em Journal of Optimization Theory and Applications},
  167:1070--1094, 2015.

\bibitem{Oksendal1}
N.~Agram, B.~Øksendal, and S.~Yakhlef.
\newblock Optimal {C}ontrol of {F}orward-{B}ackward {S}tochastic {V}olterra
  {E}quations.
\newblock {\em Non-linear Partial Differential Equations, Mathematical Physics,
  and Stochastic Analysis: The Helge Holden Anniversary Volume}, pages 3--36,
  2009.

\bibitem{Filippov}
J.P. Aubin and H.~Frankowska.
\newblock {\em Set-valued analysis}.
\newblock Modern Birkhäuser Classics, 1990.

\bibitem{Bonaccorsi}
S.~Bonaccorsi and F.~Confortola.
\newblock Optimal control for stochastic {V}olterra equations with
  multiplicative {L}évy noise.
\newblock {\em Nonlinear differential equations and applications}, pages 1--26,
  2020.

\bibitem{cannarsadaprato}
P.~Cannarsa and G.~Da~Prato.
\newblock Second-order {H}amilton--{J}acobi equations in infinite dimensions.
\newblock {\em SIAM Journal on Control and Optimization}, 29(2):474--492, 1991.

\bibitem{Note}
C.~Cuchiero and G.~Di~Nunno.
\newblock Notes - {M}arkovian lifts.
\newblock {\em to appear}, 2019.

\bibitem{CTMulti}
C.~Cuchiero and J.~Teichman.
\newblock Markovian lifts of positive semidefinite affine {V}olterra-type
  processes.
\newblock {\em Decisions in Economics and Finance}, 42:407--448, 2019.

\bibitem{CT}
C.~Cuchiero and J.~Teichman.
\newblock Generalized {F}eller processes and {M}arkovian lifts of stochastic
  {V}olterra processes: the affine case.
\newblock {\em Journal of Evolution Equations}, pages 1--48, 2020.

\bibitem{DaPrato_Soluzione}
G.~Da~Prato and J.~Zabczyk.
\newblock {\em Stochastic equations in infinite dimensions}.
\newblock Encyclopedia of Mathematics and its Applications. Cambridge
  University Press, 2 edition, 2014.

\bibitem{fiacco}
G.~Di~Nunno, A.~Fiacco, and E.~Hove~Karlsen.
\newblock On the approximation of {L}évy driven {V}olterra processes and their
  integrals.
\newblock {\em Journal of Mathematical Analysis and Applications},
  476:120--148, 2019.

\bibitem{DNG1}
G.~Di~Nunno and M.~Giordano.
\newblock Stochastic {V}olterra equations with time-changed {L}évy noise and
  maximum principles.
\newblock {\em Annals of Operations Research}, 2023.

\bibitem{GF}
G.~Fabbri, F.~Gozzi, and A.~Swiech.
\newblock {\em Stochastic Optimal Control in Infinite Dimension Dynamic
  Programming and {HJB} Equations}.
\newblock Springer, 2017.

\bibitem{FT}
M.~Fuhrman and G.~Tessitore.
\newblock Nonlinear {K}olmogorov equations in infinite dimensional spaces: the
  backward stochastic differential equations approach and applications to
  optimal control.
\newblock {\em The Annals of Probability}, 30:1397--1465, 2002.

\bibitem{Antony}
M.~Giordano and A.~Yurchenko-Tytarenko.
\newblock Optimal control in linear stochastic advertising models with memory.
\newblock {\em ArXiv}, 2021.

\bibitem{GM2006}
F.~Gozzi and C.~Marinelli.
\newblock Stochastic optimal control of delay equations arising in advertising
  models.
\newblock In {\em Stochastic Partial Differential Equations and Applications},
  Lecture Notes in Pure and Applied Mathematics, page 133–148. Chapman and
  Hall/CRC, 2005.

\bibitem{GMS2009}
F.~Gozzi, C.~Marinelli, and S.~Savin.
\newblock On controlled linear diffusions with delay in a model of optimal
  advertising under uncertainty with memory effects.
\newblock {\em Journal of optimization theory and applications},
  142(2):291–321, 2009.

\bibitem{Possamai}
C.~Hernandez and D.~Possamaï.
\newblock A unified approach to well-posedness of type-{I} backward stochastic
  {V}olterra integral equations.
\newblock {\em ArXiv}, 2020.

\bibitem{UMD}
T.~Hytönen, J.~Van~Neerven, M.~Veraar, and L.~Weis.
\newblock {\em Analysis in Banach Spaces}, volume I: Martingales and
  Littlewood-Paley Theory.
\newblock Springer, Cham, 2016.

\bibitem{LZh2020}
C.~Li and W.~Zhen.
\newblock Stochastic optimal control problem in advertising model with delay.
\newblock {\em Journal of Systems Science and Complexity}, 33:968–987, 2020.

\bibitem{Clark-Ocone}
J.~Maas and J.~Van~Neerven.
\newblock A {C}lark-{O}cone formula in {UMD} {B}anach spaces.
\newblock {\em Electronic communications in probability}, 2008.

\bibitem{Masiero}
F.~Masiero.
\newblock Stochastic optimal control problems and parabolic equations in
  {B}anach spaces.
\newblock {\em SIAM Journal on Control and Optimization}, 47:251--300, 2008.

\bibitem{MalliavinBanach}
M.~Pronk and M.~Veraar.
\newblock Tools for {M}alliavin calculus in {UMD} {B}anach spaces.
\newblock {\em Potential Analysis}, 40:307–344, 2014.

\bibitem{electrodynamics}
J.~Pr{\"u}ss.
\newblock {\em Evolutionary integral equations and applications}, volume~87.
\newblock Birkh{\"a}user, 2013.

\bibitem{epidemiological}
M.~Saeedian, M.~Khalighi, N.~Azimi-Tafreshi, G.R. Jafari, and M.~Ausloos.
\newblock Memory effects on epidemic evolution: The
  susceptible-infected-recovered epidemic model.
\newblock {\em Physical Review E}, 95(2):022409, 2017.

\bibitem{VanNerv}
J.~Van~Neerven.
\newblock {\em Stochastic evolution equations}.
\newblock Lecture notes, 2007.

\bibitem{Fubini}
J.~Van~Neerven and M.~Veraar.
\newblock On the stochastic {F}ubini theorem in infinite dimensions.
\newblock In {\em In Stochastic partial differential equations and
  applications—VII, volume 245 of Lect}, 2005.

\bibitem{BDG}
J.~Van~Neerven, M.~Veraar, and L.~Weis.
\newblock Stochastic integration in {UMD} {B}anach spaces.
\newblock {\em The Annals of Probability}, 35, 2007.

\bibitem{Yong1}
J.~Yong.
\newblock Backward {S}tochastic {V}olterra {I}ntegral {E}quations and some
  {R}elated {P}roblems.
\newblock {\em Stochastic Processes and their Applications}, 116:770--795,
  2006.

\bibitem{Yong2}
J.~Yong.
\newblock Well-{P}osedness and {R}egularity of {B}ackward {S}tochastic
  {V}olterra {I}ntegral {E}quations.
\newblock {\em Probability Theory and Related Fields}, 142:21--77, 2007.

\end{thebibliography}

\end{document}